\documentclass[12pt]{amsart}
\usepackage{amsfonts, amsmath, amsthm}
\usepackage{graphicx}
\usepackage{fourier,microtype}
\usepackage{geometry}
\usepackage{eepic}

\newtheorem{thm}{Theorem}[section]
\newtheorem{conj}[thm]{Conjecture}
\newtheorem{cor}[thm]{Corollary}
\newtheorem{prop}[thm]{Proposition}
\newtheorem{lem}[thm]{Lemma}

\theoremstyle{definition}
\newtheorem{dfn}[thm]{Definition}
\newtheorem{notation}[thm]{Notation}
\newtheorem{dfns}[thm]{Definitions}

\theoremstyle{remark}
\newtheorem{rmk}[thm]{Remark}

\def\vol[#1]{\mbox{\rm Vol}(#1)}

\def\lv[#1]{\mbox{\rm LinkVol}(#1)}
\def\lvd[#1#2]{\mbox{\rm LinkVol}_{#1}(#2)}
\def\lva[#1]{\mbox{\rm LinkVol}_{{X},d}(#1)}
\def\lvsd[#1#2]{\mbox{\rm LinkVol}_{s,{#1}}(#2)}

\def\kv[#1]{\mbox{\rm KnotVol}(#1)}
\def\kvd[#1#2]{\mbox{\rm KnotVol}_{#1}(#2)}
\def\kvs[#1]{\mbox{\rm KnotVol}_{s}(#1)}
\def\kvsd[#1#2]{\mbox{\rm KnotVol}_{s,{#1}}(#2)}

\def\wt{\widetilde}

\def\depth[#1]{\mbox{\rm depth}(#1)}

\title{The Link Volume of 3-Manifolds}

\date{\today}  
\address{Department of mathematical Sciences, University of Arkansas, 
Fayetteville, AR 72701} %
\address{Department of Information and Computer Sciences, Nara Women's
  University  Kitauoya Nishimachi, Nara 630-8506, Japan} %
\email{yoav@uark.edu} %
\email{yamasita@ics.nara-wu.ac.jp} %
\author{Yo'av Rieck} %
\author{Yasushi Yamashita} %


\begin{document}
\begin{abstract}
We view closed orientable 3-manifolds as covers of $S^{3}$
branched over hyperbolic links.  For a cover $M \stackrel{p}{\to} S^{3}$,
of degree $p$ and branched over a hyperbolic link $L \subset S^{3}$,
we assign the complexity $p \vol[S^{3} \setminus L]$.  We define an invariant of 3-manifolds,
called the {\it link volume} and denoted $\lv[M]$, that assigns to a 3-manifold $M$ the 
infimum of the complexities
of all possible covers $M \to S^{3}$, where the only constraint is that
the branch set is a hyperbolic link. 
Thus the link volume measures how efficiently $M$ can be represented as a 
cover of $S^3$.  

We study the basic properties of the link volume and related invariants,
in particular observing that for any hyperbolic manifold $M$, $\vol[M] < \lv[M]$.
We prove a structure theorem (Theorem~\ref{thm:jt}) that is similar to (and
uses) the celebrated theorem of J\o rgensen and Thurston.  This leads us to
conjecture that,
generically, the link volume of a hyperbolic 3-manifold 
is much bigger than its  volume (for precise statements
see Conjectures \ref{conj:LV>>V1} and \ref{conj:LV>>V2}).

Finally we prove that the link volumes of the manifolds obtained by Dehn filling a manifold with boundary
tori are linearly bounded above in terms of the length of the continued fraction expansion of the filling
curves (for a precise statement, see Theorem \ref{thm:dehn}).
\end{abstract}

\nocite{*}

\maketitle

\section{Introduction}
\label{sec:intro}

The study of 3-manifolds as branched covers of $S^{3}$ has a long history.
In 1920 Alexander~\cite{alex} gave a very simple argument showing
that every closed orientable triangulated 3-manifold is a 
cover of $S^{3}$ branched along the 1-skeleton of a tetrahedron
embedded in $S^{3}$.  
We explain his construction and give basic definitions in 
Section~\ref{sec:background}.  Clearly, if a 3-manifold $M$ is a
finite sheeted branched cover of $S^{3}$, then $M$ is closed and orientable.
Moise~\cite{moise} showed that every closed 3-manifold admits a triangulation;
thus we see: a 3-manifold $M$ is
closed and orientable if and only if $M$ is a finite sheeted branched cover of $S^{3}$.
From this point on, by {\it manifold} we mean connected closed orientable 3-manifold.

Alexander himself noticed one weakness
of his theorem: the branch set is not a submanifold.  He claimed that this can
be easily resolved, but gave no indication of the proof.
In 1986 Feighn~\cite{feighn} substantiated Alexander's claim, 
Modifying the branch set to be a link.

Thurston showed the existence of a {\it universal
link}, that is, a link $L \subset S^{3}$ so that every 3-manifold is a cover of $S^{3}$
branched along $L$.  Hilden, Lozano and Montesinos~\cite{hlm1}~\cite{hlm2} drastically simplified 
Thurston's example showing, in particular, that the figure eight knot is
universal.  Cao and Meyerhoff~\cite{CaoMeyehoff} showed that the
figure eight knot is the hyperbolic link of smallest 
volume.  In this paper, we 
consider hyperbolic links and consider their volume as a measure
of complexity, hence we see that every 3-manifold is a 
cover of $S^{3}$, branched along the simplest possible link.

Our goal is to define and study invariant that asks: how efficient is
the presentation of a 3-manifolds as a branched over of $S^{3}$?
We do this as follows: 
let $M$ be a $p$-fold cover of $S^{3}$, branched 
along the hyperbolic link $L$.  We denote this as $M \stackrel{p}{\to} (S^3,L)$
(read: $M$ is a $p$-fold cover of $S^{3}$ branched along $L$).
The complexity of $M \stackrel{p}{\to} (S^3,L)$  is defined to be
the degree of the cover times the volume of $L$, that is:
$$p \vol[S^3 \setminus L].$$ 

The {\it link volume} of $M$, denoted $\lv[M]$, is the infimum of the complexities of all
covers $M \stackrel{p}{\to} (S^3,L)$, subject to the constraint that $L$ is a hyperbolic link; that is:
$$\lv[M] = \inf \{ p \vol[S^3 \setminus L] | M \stackrel{p}{\to} (S^3,L)\, ; \, L \mbox{ hyperbolic}\}.$$

Given a hyperbolic manifold $M$ we consider its volume, $\vol[M]$, as
its complexity.  This is consistent with our attitude towards hyperbolic
links, and is considered very natural by many 3-manifold
topologists.  Why is that?  What is it that the volume
actually measures?  Combining results of Gromov, J\o rgensen, and 
Thurston (for a detailed exposition see~\cite{KR}) 
we learn the following.  Let $t_C(M)$ denote the minimal number of 
tetrahedra required to triangulate a link exterior in $M$, that is, the least 
number of tetrahedra required to triangulate $M \setminus N(L)$, 
where the minimum is taken over all possible links $L \subset M$ 
(possibly, $L = \emptyset$) and
all possible tringulations of $M \setminus N(L)$.
Then there exist constants $a, b>0$ so that
\begin{equation}
a \vol[M] \leq t_C(M) \leq b\vol[M].
\end{equation}
We consider invariants up-to linear equivalence, and so we see that 
$\mbox{Vol}$ and $t_C$ are equivalent.  
This gives a natural, topological interpretation of
the volume.   In this paper we begin the study of the link volume, with the 
ultimate goal of obtaining a topological understanding of it.

The basic facts about the link volume are presented in Section~\ref{sec:basic}.
The most important are the following easy observations:  
	\begin{enumerate} 
	\item The link volume is obtained, that is, for any manifold $M$ there is a 
	cover $M \stackrel{p}{\to} (S^3,L)$ so that $\lv[M] = p \vol[S^{3} \setminus L]$.
	\item For every hyperbolic 3-manifold $M$ we have:
	$$\vol[M] < \lv[M].$$
	\end{enumerate}
The second point begs the question: is the link volume of hyperbolic manifolds 
equivalent to the hyperbolic 
volume?  As we shall see below, the results of this paper lead us to believe
that this is not the case (Conjectures~\ref{conj:LV>>V1} and~\ref{conj:LV>>V2}).

The right hand side of the Inequality~(1) implies that, for fixed $V$,
any hyperbolic manifold of volume less than $V$ can be obtained from
a manifold $X$ by Dehn filling, where $X$ is constructed using at most
$bV$ tetrahedra.  Since there are only finitely many such $X$'s, 
this implies the celebrated 
result of J\o rgensen--Thurston: for any $V>0$, there exists finite collection of
compact ``parent manifolds'' $\{X_i,\dots X_{n}\}$, so that $\partial X_{i}$ consists of tori,
and any hyperbolic manifold of
volume at most $V$ is obtained by Dehn filling $X_{i}$, for some $i$.
Our first result is:
\begin{thm}
\label{thm:jt}
There exists a universal constant $\Lambda > 0$ so that
for every $V>0$, there is a finite collection $\{\phi_{i}:X_{i} \to E_{i}\}_{i=1}^{n_{V}}$,
where $X_{i}$ and $E_{i}$ are complete finite volume hyperbolic manifolds
and $\phi_{i}$ is an unbranched cover, and for any cover  
$M \stackrel{p}{\to} (S^3,L)$ with $p \vol[S^{3} \setminus L] < V$
the following hold:
	\begin{enumerate}
	\item For some $i$, $M$ is obtained from $X_{i}$ by Dehn filling, $S^{3}$ is 
	obtained from $E_{i}$ by Dehn filling, and the following diagram commutes 
	(where the vertical arrows represent the covering
	projections and the horizontal arrows represent Dehn fillings):	
\begin{center}
\begin{picture}(200,60)(0,0)
  \put(  0,  0){\makebox(0,0){$E_{i}$}}
  \put(  0,50){\makebox(0,0){$X_{i}$}}
  \put(  0, 40){\vector(0,-1){30}}   
  \put( 10,  0){\vector(1,0){75}}
  \put(100,  0){\makebox(0,0){$S^3,L$}}
  \put(100, 40){\vector(0,-1){30}}
  \put(10,28){\makebox(0,0){$/\phi_{i}$}}
  \put(108,28){\makebox(0,0){$/\phi$}}
  \put(108,28){\makebox(0,0)}
  \put(100,50){\makebox(0,0){$M$}}
  \put( 10,50){\vector(1,0){80}}
\end{picture}
\end{center}
	\item $E_{i}$ can be triangulated
	using at most $\Lambda V/p$ tetrahedra (hence $X_{i}$ can be triangulated 
	using at most $\Lambda V$ tetrahedra and $\phi_{i}$ is simplicial).
	\end{enumerate}
\end{thm}

For $V>0$, let $\mathcal{M}_{V}$ denote the set of manifolds of 
link volume less than $V$.  Since the link volume is always
obtained, applying Theorem~\ref{thm:jt} to covers realizing the link 
volumes of manifolds in $\mathcal{M}_{V}$,
we obtain a finite family of ``parent manifolds'' $X_{1},\dots,X_{n}$
that give rise to every manifold in $\mathcal{M}_{V}$ via Dehn filling, much like
J\o rgensen--Thurston.  The 
extra structure given by the projection $\phi_{i}:X_{i} \to E_{i}$ implies
that the fillings that give rise to manifolds of low link volume are very
special:

Fix $V$, and let $X_i$ be as in the statement of Theorem~\ref{thm:jt}.  Then for any
hyperbolic manifold $M$ that is obtained by filling $X_i$ we have
$\vol[M] < \vol[X_i]$.  On the other hand, it is by no means clear 
that $\lv[M] < V$, for it is not easy to complete the diagram in 
Theorem~\ref{thm:jt}:
	\begin{enumerate}
	\item $X_i$ must cover a manifold $E_i$.
	\item The covering projection and the filled slopes
	must be compatible (see Subsection~\ref{subsec:slopes} for definition).
	\item The slopes filled on $E_i$ must give $S^3$, a very unusual situation since $E_i$ is hyperbolic.
	\end{enumerate}
These lead us to believe that the link volume, as a fuction, is much bigger
than the volume.  Specifically we conjecture:

\begin{conj}
\label{conj:LV>>V1}
Let $X$ be a complete finite volume hyperbolic manifold with one cusp.  
For a slope $\alpha$ on $\partial X$, let $X(\alpha)$ denote
the closed manifold obtained by filling $X$ along $\alpha$.

Then for any $V>0$, 
there exists a finite set of slopes $\mathcal{F}$ on $\partial X$,
so that if $\lv[X(\alpha)] < V$, then $\alpha$ intersects some slope in 
$\mathcal{F}$ at most $V/2$
times.
\end{conj}

As is well known, the volume of the figure eight knot complement is about $2.029\dots$, 
twice $v_{3}$,  the volume of a regular ideal tetrahedron.  
By considering manifolds that are obtained by
Dehn filling the figure eight knot exterior we see that
Conjecture~\ref{conj:LV>>V1} implies:

\begin{conj}
\label{conj:LV>>V2}
For every $V>0$ there exists a manifold $M$ so that $\vol[M] < 2v_{3} = 2.029\dots$ 
and $\lv[M] > V$.
\end{conj}

To describe our second result, we first define the knot volume and a few other variations of the link volume; for the definition
simple cover see the Subsection~\ref{subsec:monte}.

\begin{dfns}
\label{dfns:KV}
	\begin{enumerate}
	\item The {\it knot volume} of a 3-manifold $M$ is obtained by considering only hyperbolic knots in the 
	definition of the link volume,
	that is, 
	$$\kv[M] = \inf \{ p \vol[S^3 \setminus K] | M \stackrel{p}{\to} (S^3,K); K \mbox{ is a hyperbolic knot}\}.$$
	\item The {\it simple knot volume} of a 3-manifold $M$ is obtained by considering only simple 
	covers in the definition of the knot volume,
	that is, 
	$$\kvs[M] = \inf \left\{ p \vol[S^3 \setminus K] \Bigg| M \stackrel{p}{\to} (S^3,K); \begin{array}{l} K 
	\mbox{ a hyperbolic knot}, \\ \mbox{and the cover is simple} \end{array} \right\}.$$
	\item For an integer $d \geq 3$, the {\it simple $d$-knot volume} in obtained by 
	restricting to $p$-fold covers for $p \leq d$ in the definition of the simple knot volume, that is,
	$$\kvsd[d M] = \inf \left\{ p \vol[S^3 \setminus K] \Bigg| M \stackrel{p}{\to} (S^3,K); \begin{array}{l} K 
	\mbox{ a hyperbolic knot}, \\ \mbox{the cover is simple}, \\ \mbox{and } p \leq d \end{array} \right\}.$$
	\end{enumerate}

\end{dfns}

Similarly, one can play with various restrictions on the covers considered.  However, one must 
ensure that the definition makes sense.  For example, the {\it regular}
link volume can be defined using only regular covers.  This makes no sense, as not every 
manifold is the regular cover of $S^{3}$.  It follows from Hilden~\cite{hilden} 
and Montesinos~\cite{montesinos}
that every 3-manifold is a simple 3-fold cover of $S^{3}$ branched over 
a hyperbolic knot; hence the definitions above make sense.  
Our next result is an upper bound, and holds for any of the variations listed in Definitions~\ref{dfns:KV}.  
Since these definitions are obtained by adding restrictions
to the covers considered, it is clear that $\kvsd[3 M]$ is greater than or equal to any of the others,
including the link volume.  
We therefore phrase Theorem~\ref{thm:dehn} below
for that invariant.  But first we need:

\begin{dfn}
\label{dfn:depth}
Let $T$ be a torus, and $\mu$, $\lambda$ generators for $H_{1}(T)$.  By identifying $\mu$ with $1/0$ and $\lambda$
with $0/1$, we get an identification of the {\it slopes} of $H_1(T)$ with $\mathbb{Q} \cup \{1/0\}$,
where an element of $H_{1}(T)$ is called a {\it slope} if it
can be represented by a connected simple closed curve on $T$.  
Then the {\it depth} of a slope $\alpha$, 
denoted $\depth[\alpha]$, is the length of
the shortest contiuded fraction expenssion representing $p/q$.  For a collection of tori $T_{1},\dots,T_{n}$
with bases chosen for $H_{1}(T_{i})$ for each $i$, we define 
$$\depth[p_{1}/q_{1},\dots,p_{n}/q_{n}] = \Sigma_{i=1}^{n} \depth[p_{i}/q_{i}].$$
\end{dfn}

We are now ready to state:

\begin{thm}
\label{thm:dehn}
Let $X$ be a connected, compact orientable $3$-manifold, $\partial X$ consisting of $n$ tori
$T_{1},\dots,T_{n}$, and fix $\mu_{i}$, $\lambda_{i}$, generators for $H_{1}(T_{i})$ for each $i$.

Then there exist a universal constant $B$ and a constant $A$ that depends on $X$ and the choice of
bases for $H_{1}(T_{i})$, so that for any $p_{i}/q_{i}$ ($i=1,\dots,n$),
$$\kvsd[3 X(p_{1}/q_{1},\dots,p_{n}/q_{n})] < A + B \depth[p_{1}/q_{1},\dots,p_{n}/q_{n}],$$
where $X(p_{1}/q_{1},\dots,p_{n}/q_{n})$ denotes the manifold obtained by filling $X$ along
the slopes $p_{i}/q_{i}$.
\end{thm}

As noted above, $\kvsd[3 M]$ is greater than or equals to all the invariants defined in Definition~\ref{dfn:depth}
and the link volume.
Hence Theorem~\ref{thm:dehn}, which gives an upper bound, holds for all these invariants, and in particular:

\begin{cor}
\label{cor:dehn}
With the hypotheses of Theorem~\ref{thm:dehn},
there  exist a universal constant $B$ and a constant $A$ that depends on $X$ and the choice of
bases for $H_{1}(T_{i})$, so that for any slopes $p_{i}/q_{i}$ ($i=1,\dots,n$),
$$\lv[X(p_{1}/q_{1},\dots,p_{n}/q_{n})] \leq A + B \depth[p_{1}/q_{1},\dots,p_{n}/q_{n}].$$
\end{cor}

\bigskip

\noindent{\bf Organization.} This paper is organized as follows.  In Section~\ref{sec:background} 
we go over necessary background material.  
In Section~\ref{sec:variations} we explain some possible variation on the link volume.  Notably,
we define the {\it surgery volume} (definition due to Kimihiko Motegi) 
and an invariant denote $pB(M)$ (definition due to Ryan Blair). We show that
{\it ,in contrast to the link volume, }
the surgery volume of hyperbolic manifolds is bounded in terms of their volume.
We also show that $pB(M)$ is linearly equivalent to $g(M)$, the Heegaard genus of $M$.
In Section~\ref{sec:basic} we explain basic
facts about the link volume and list some open questions.
In Section~\ref{sec:jt} we prove Theorem~\ref{thm:jt}.
In Section~\ref{sec:dehn} we prove Theorem~\ref{thm:dehn}.

\bigskip

\noindent{\bf Acknowledgement.}  We thank Ryan Blair, Tsuyoshi Kobayashi, 
Kimihiko Motegi, 
Hitoshi Murakami, and Jair Remigio--Ju\'arez for helpful conversations.

\section{Background}
\label{sec:background}

By {\it manifold} we mean connected, closed, orientable 3-manifold.  In some cases, we consider 
connected, compact, orientable 3-manifolds; then we explicitly say {\it compact manifold}.
By {\it hyperbolic manifold} $X$ we mean a complete, finite volume 
Riemannian 3-manifold locally isometric to $\mathbb{H}^{3}$.  It is well know
that any hyperbolic manifold $X$ is the interior
of a compact manifold $\bar{X}$ and $\bar{X} \setminus X = \partial \bar{X}$ consists
of tori.  To simplify notation, we do not refer to $\bar{X}$ explicitly and call $\partial \bar X$
the {\it boundary of} $X$.  We assume familiarity
with the basic concepts of 3-manifold theory and hyperbolic manifolds, and in particular
the Margulis constant.  By {\it volume} we mean the hyperbolic volume.
The volume of a hyperbolic manifold $M$ is denoted $\vol[M]$.

We follow standard notation.  In particular, by {\it Dehn filling} (or simply {\it filling}) we mean 
attaching a solid torus to a torus boundary component.  

\subsection{Branched covering} 
\label{subsec:covers}
We begin by recalling Alexander's Theorem~\cite{alex}; because this theorem is very
short an elegant, we include a sketch of its proof here.

\begin{thm}[Alexander]
Let $\mathcal{T}$ be a triangulation of $S^{n}$ obtained by doubling an $n$-simplex.
Let $M$ be a closed orientable triangulated $n$-manifold.  Then $M$ is a cover
of $S^{n}$ branched along $\mathcal{T}^{(n-2)}$, the $n-2$-skeleton of $\mathcal{T}$.
\end{thm}

\begin{proof}[Sketch of Proof]
Let $M$ be as above.
Given $\mathcal{T}_M$, a triangulation of $M$, let $\mathcal{T'}_M$ denote its barycentric subdivision.  Each vertex $v$ 
of $\mathcal{T}_M'$ is the center of a $k$-face of $\mathcal{T}_M$, for some $k$.  Label $v$ with the label $k$.
By construction, there are exactly $n+1$ labels, $0,\dots,n$, and no two adjacent vertices have the 
same label.

Note that the 1-skeleton of $\mathcal{T}$
is $K_{n+1}$, the complete graph on $n+{1}$ vertices.  Label these vertices with the labels
$0,\dots,n$ so that every label appears exactly once.

We define a function from  $\mathcal{T}_M'^{(n-1)}$ (the $n-1$ skeleton of  $\mathcal{T}_M'$) to $S^{n}$ by sending each $k$-face
simplicially to the unique $k$-face of $S^{n}$ with the same labeling (for $k < n$); it is easy to see that this function is well defined.
However, the $n$-cells of 
$M$ can be sent to either of the two $n$ simplices of $\mathcal{T}_M'$.  We pick the simplex so that the map is
orientation preserving.

It is left to the reader to verify that this is indeed a cover, branched over the $n-2$ skeleton
of the triangulation of $S^{n}$.
\end{proof}

\begin{lem}
\label{lem:finiteness}
For any compact triangulated $n$-manifold $M$, $B \subset M^{(n-2)}$ a subcomplex, and $d>0$, 
there are only finitely many $d$-fold covers of $M$ branched along $B$.
\end{lem}

\begin{proof}
It is well known that a $p$-fold cover of  $M$ branched along $B$ is determined by a presentation of $\pi_{1}(M \setminus B)$
into $S_{p}$, the symmetric group on $p$ elements (see, for example, \cite{rolfsen}).
The lemma follows from the fact that $\pi_{1}(M\setminus B)$ 
is finitely generated and $S_{p}$ is finite.
\end{proof}

\subsection{Simple covers and the Montesinos Move}
\label{subsec:monte}

\begin{dfn}
Let $f:M \to N$ be a cover of finite degree $p$ branched along $B \subset N$.  Note that every point
of $N \setminus B$ has exactly $p$ preimages, and every point of $B$ has at most $p$ preimages.
$f:M \to N$ is called {\it simple} if every point of $B$ has exactly $p-1$ preimages.
\end{dfn}

Let $M \to (S^{3},L)$ be a 3-fold simple cover branched along the link $L$.  We view $L$ diagrammatically, as projected
into $S^{2} \subset S^{3}$ in the usual way.  Since the cover is simple, each generator in the Wirtinger presentation
of $S^3 \setminus L$ corresponds to a permutation in the symmetric group on 3 elements (that is, $(1 \ 2)(3)$ or
$(1 \ 3)(2)$ or $(2 \ 3)(1)$).  We consider these as three colors, and color each strand of $L$ accordingly.
By assumption, $M$ is connected; hence not all generators correspond to the same
permutation.  Finally, the relators of the Wirtinger presentation guarantee that 
at each crossing wither all three color appear, or only one color does.  Thus we obtain a 3 coloring of the strands of $L$.

Montesinos proved that if we replace a positive crossing where {\it all three colors appear}
by 2 negative crossings the cover is not changed.  
This is called the {\it Montesinos move}.  The reason is simple:
the neighborhood of a 3-colored crossing is a ball, and its cover is a ball as well.  
(This is false if only one color appears at the crossing!)  More generally, when all three colors appear 
we can replace $n$ half twists with $n + 3k$ half twists ($n,k \in \mathbb Z$).
The case $n=0$ is allowed, but then we must require that the two strands in question have distinct colors.
We denote such a move by $n \mapsto n+3k$ Montesinos move.  
In Figure~\ref{fig:montesinos} we show a few views of the Montesinos Move.
\begin{figure}
\includegraphics{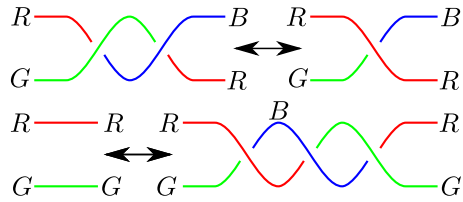}
\caption{Montesinos move}
\label{fig:montesinos}
\end{figure}

Finally, we record the following fact for future reference.
It is easy to see that the $p$-fold cover cover $f:M \to S^{3}$ branched along
$B \subset S^{3}$ is connected if and only if the image of $\pi_{1}(S^{3} \setminus B)$
in $S_{p}$ acts transitively on the set f $p$ letters.  For simple 3-fold covers this means:

\begin{lem}
\label{lem:connected3fold}
Let $M$ be a 3-manifold and $f:M \to S^{3}$ a simple 3-fold cover branched along the 
link $L \subset S^{3}$.  Then $M$ is connected if and only if at least two colors appear in
the 3-coloring of $L$. 
\end{lem}

\subsection{Slopes on tori and coverings}
\label{subsec:slopes}

Recall that s {\it slope} on a torus is the free homotopy class of a connected simple closed curve, 
up to reserving the orientation of the curve.  For this subsection we fix the following:
let $X$ and $E$ be complete hyperbolic manifolds of finite volume, and $\phi:X \to E$
an unbranched cover. Let $T$ be a boundary component of $X$; note that $\phi$
induces an unbranched cover $T \to \phi(T)$.

Let $\alpha$ be a slope on $T$ realized by a connected simple
closed curve $\gamma \subset T$.  Then $\phi(\gamma)$
is a (not necessarily simple) connected essential curve on $\phi(T)$.  Since $\phi(T)$ is a torus, there is a 
curve $\bar\beta$
on $\phi(T)$ so that $\phi(\gamma)$ is homotopic to $\bar\beta^{m}$, for some $m \neq 0$.
Let $\beta$ be the slope defined by $\bar\beta$.
Define the function $\phi_{\downarrow}$ from the slopes on $T$ to the slopes on $\phi(T)$ by
setting $\phi_{\downarrow}(\alpha) = \beta$.

Conversely, let $\alpha$ be a slope on $\phi(T)$ realized by a connected simple closed curve 
$\gamma \subset \phi(T)$.  Then $\phi^{-1}(\gamma)$ is a (not necessarily connected)
essential simple closed curve.  Each component of $\phi^{-1}(\gamma)$ defines a slope on $T$, 
and since these curves are disjoint, they
all define the same slope, say $\beta$.  Define the function $\phi_{\uparrow}$ from the slopes on 
$\phi(T)$ to the slopes on $T$ by setting  $\phi_{\uparrow}(\alpha) = \beta$.
It is easy to see that $\phi_{\downarrow}$ is the inverse of $\phi_{\uparrow}$.
We say that $\alpha$ and $\phi_{\downarrow}(\alpha)$ are {\it corresponding} slopes.

Suppose that we Dehn fill $T$ and $\phi(T)$.  If the slope filled are not corresponding, then the curve 
filled on $T$ maps to a curve of $\phi(T)$ that is not null homotopic in the attached solid torus.
Thus the map $\phi$ cannot be extended into that solid torus.  

Conversely, suppose that corresponding slopes are filled.  We parametrize the attached solid
tori as $S^{1} \times D^{2}$, and extend $\phi$ into the  solid tori by coning along each 
disk $\{p\} \times D^{2}$ ($p \in S^{1}$).  It is easy to see that the extended map is a cover,
branched (if at all) along the core of the attached solid torus.  (The local degree at the core of the
solid torus is the number denoted by $m$ in the construction of $\phi_{\downarrow}$ above.)

In conclusion, $\phi$ induces a correspondence between slopes of $T$ and 
slopes on $\phi(T)$, and $\phi$
can be extended to the attached solid tori to give a branched cover after 
Dehn filling if and only if corresponding slopes are filled.

Next, let $T_{1}$, $T_{2} \subset \partial X$ be tori that project to the same 
component of $\partial E$.  Then two bijections $\phi_{\downarrow}$ from
the slopes of $T_{1}$ and $T_{2}$ to the slopes of $\phi(T_{1}) = \phi(T_{2})$
induce a bijection between the slopes of $T_{1}$ and the slopes of
$T_{2}$; again we call slopes that are interchanged by this bijection
{\it corresponding}.  Filling $T_{1}$ and $T_{2}$ along corresponding slopes
is called {\it consistent}, {\it inconsistent} otherwise.  Note that 
after filling $X$ there is a filling of $E$ so that the cover $X \to E$
extends to a branched cover if and only if the filling of $X$ is consistent on
every pair of components of $\partial X$.

\subsection{Hyperbolic alternating links}  
\label{sebsection:alternating}

In this subsection we follow Chapter~4 of Lickorish \cite{lickorish}.
We begin with the following standard definitions:

\begin{dfns}
\label{dfn:alternating}
Let $L$ be a link and $D$ a diagram for $L$.  The projection sphere is
denoted $S^{2}$.  Then $D$ is called
{\it alternating} if, for each component $K$ of $L$, when traversing
the projection of $K$ the crossing occur as \dots over, under, over, 
under,\dots.  $L$ is called an {\it alternating link} if it admits an
alternating diagram.  A link diagram $D$ is called {\it strongly prime}
if any simple closed curve that intersects it transversely in two simple points
(that is, two points that are not crossings) bounds a disk that $D$ 
intersects in a single arc with no crossings.
A link $L$ is called {\it split} if its exterior admits an essential sphere,
that is, if there is an embedded sphere $S \subset S^3 \setminus L$ so that
each of the balls obtained by cutting $S^3$ open along $S$ contains 
at least one component of $L$.  A link diagram $D \subset S^2$ is called a {\it split
diagram} if there is a circle $\gamma$ embedded in $S^2$, so that each
disk obtained by cutting $S^2$ open along $\gamma$ contains at least
one component of $D$.  Note that a split diagram is necessarily a diagram
for a split link, but the converse does not hold.  A link is called
{\it simple} if its exterior does not admit an essential surface of non-negative
Euler characteristic.  A link $L$ is called {\it hyperbolic} if $S^3 \setminus L$
admits a complete, finite volume, hyperbolic metric.
\end{dfns}

Menasco (\cite{menasco}, see also~\cite{lickorish}) proved:

\begin{thm}
\label{thm:menasco}
Let $D$ be an alternating link diagram for a link $L$.  If $D$ is 
strongly prime and is not split, then $L$ is simple.
\end{thm}

Thurston proved:

\begin{thm}
\label{thm:thurston}
Any simple link is hyperbolic.
\end{thm}

Combining these results, we obtain:

\begin{cor}
\label{cor:hyperbolic}
If a link $L$ has a non-split, strongly prime, alternating diagram,
then $L$ is hyperbolic.
\end{cor}

\subsection{Twist number and hyperbolic volume}
\label{subsec:twist}

For the definition of twist number see, for example,~\cite{lackenby}.
We briefly recall it here. 
Let $D$ be a link diagram.  Let $\sim$ be the equivalent relation on the 
crossings of $D$ generated by $c \sim c'$ if $c$ and $c'$ lie on the boundary
of a bigon of $D$.  This equivalence 
relation can be visualized as follows: if $c_{1},\dots,c_{n}$ form an 
equivalence class of crossings, then after reordering them if necessary,
there is a chain of $n-1$ bigons in $D$ with $c_{i-1}$ and $c_{i}$
on the boundary the $i$th bigon.

The {\it twist number} of a link $L$, denoted $t(L)$, is the 
smallest number of equivalence classes in any diagram for $L$.  
Thus, for example, the obvious diagram of
twist knots show they have twist number at most 2.

Lackenby~\cite{lackenby} gave upper and lower bounds on the hyperbolic
volume of link exteriors in terms of their twist number.  We emphasize that
the lower bound holds for alternating links (or, more precisely, for alternating diagrams), 
while the upper bound holds {\it for all links}.  
It is the upper bound that we will need in this work,
hence we need not assume the diagram alternates.  We will need:

\begin{thm}[Lackenby~\cite{lackenby}]
\label{lackenby}
There exists a constant $c$ so that for any hyperbolic link $L$,
$$\vol[S^{3} \setminus L] \leq c t(L).$$
\end{thm}

\section{Variations}
\label{sec:variations}

In this section we discuss two variations of the link volume.  The first variation is obtained by replacing the
volume by another knot invariant (note that one can use any invariant with values in $\mathbb R_{\geq 0}$).
This variation was suggested by Ryan Blair.  Let $L \subset S^{3}$ be a link
and let $b(L)$ denote its bridge index.  We consider the complexity of  
$M \stackrel{p}{\to} (S^3,L)$ to be $p b(L)$.  Define $\mbox{pB}(M)$ to be the infimum of $pb(L)$, taken over all possible covers.

It is easy to see that the preimage 
of a bridge surface $S$ for $L$ is a Heegaard surface for $M$, say $\Sigma$.  
Since $S$ is a $2b$ punctured sphere, $\chi(S \setminus L) = 2-2b$.  Its preimage
has Euler characteristic $p(2-2b)$.  We obtain $\Sigma$ by adding some number of points, 
say $n \geq 0$.  Then $\chi(\Sigma) = 
p (2-2b) + n$.   Thus we get:
\begin{eqnarray*}
2g(\Sigma) - 2 & = & -\chi(\Sigma) \\
	& = & p(2b-2) - n \\
	& = & 2pb -(2p + n) \\
	& \leq  & 2pb - 2.
\end{eqnarray*}
Since $\mbox{pB}(M)$ is positive integer valued, the infimum is obtained.  By considering a cover that
realizes $\mbox{pB}(M)$, we obtain a surface $\Sigma$ so that $g(\Sigma) \leq \mbox{pB}(M)$.
Thus $g(M) \leq g(\Sigma) \leq \mbox{pB}(M)$.

The converse is highly non-trivial.  Given an arbitrary manifold $M$, Hilden~\cite{hilden}
constructed a $3$-fold cover $M \stackrel{3}{\to} (S^3,L)$.  The construction uses an arbitrary
Heegaard surface $\Sigma \subset M$.  
One feature of Hilden's construction is that $b(L) \leq 2g(\Sigma) + 2$.
Since $\Sigma$ was an arbitrary Heegaard surface, we may assume that
$g(\Sigma) = g(M)$.  Thus we see that $\mbox{pB}(M) \leq 6g(M) + 6$.  
Combining the inequalities we got we obtain:
 $$g(M) \leq \mbox{pB}(M) \leq 6g(M) + 6.$$
Thus we see that the Heegaard genus and $\mbox{pB}$ are equivalent.

\bigskip

\noindent Another variation, suggested by Kimihiko Motegi, is the {\it surgery volume}.  Given a manifold $M$, 
it is well known that $M$ is obtained by Dehn surgery on a link in $S^{3}$, say $L$.  By
Myers ~\cite{myers}, every compact 3-manifold admits a simple knot.  Applying this to
$S^3 \setminus N(L)$ we obtain a knot $K$ so that $L' = L \cup K$ is a hyperbolic link.
Since $M$ is obtained from $S^{3}$ via surgery along $L'$ (with the original surgery 
coefficients on $L$ and the trivial slope on $K$), we conclude that $M$ is obtained from
$S^{3}$ via surgery along a hyperbolic link.  The surgery volume of $M$ is then 
		$$\mbox{SurgVol}(M) = \inf \{ \vol[S^3 \setminus L] | M \mbox{ is obtained 
					 by surgery on }L, L\mbox{ is hyperbolic}\}.$$

Neumann and Zagier \cite{NeumannZagier} showed that if a hyperbolic manifold $N_1$ is
obtained by filling a hyperbolic manifold $N_2$, then $\vol[N_1] < \vol[N_2]$.  
Applying this in our setting (with $S^3 \setminus L'$ as $N_1$ and $M$ as $N_2$)
we see that for any hyperbolic manifold $M$, $\vol[M] \leq \mbox{SurgVol}(M).$

We note that there exists a function $f:(0,\infty) \to (0,\infty)$ so that any hyperbolic manifold $M$ is
obtained by surgery on a hyperbolic link $L \subset S^{3}$ with $\vol[S^{3} \setminus L] \leq f(\vol[M])$.
To see this, fix $V$ and let $X_{1},\dots,X_{n}$ be the set of parent manifolds of all hyperbolic
manifolds of volume at most $V$.
For each $X_{i}$ there is a link $L_{i}$ in $S^{3}$, so that $X_{i}$ is obtained by surgery on some
of the components of $L_{i}$ and drilling the rest.  Therefore, any hyperbolic manifold $M$ with volume
at most $V$ is obtained on surgery on some $L_{i}$ ($i=1,\dots,n$).  Set 
$$f(V) = \max_{i=1}^{n}\{\vol[S^{3} \setminus L_{i}]\}.$$
We get:
$$\vol[M] \leq \mbox{SurgVol}(M) \leq f(\vol[M]).$$
The surgery volume and the hyperbolic volume are equivalent if there is a {\it linear} function
$f$ as above; we do not know if this is the case.

\section{Basic facts and open questions}
\label{sec:basic}

Basic facts about the Link Volume:

\smallskip

\begin{description}
\item[The link volume is obtained] that is, for every $M$ there exists a cover $M \stackrel{p}{\to} (S^3,L)$ so that $\lv[M] = p \vol[S^{3} \setminus L]$.
	Recall that the link volume was defined as an infimum.  To see that there is a cover realizing it, we need to show that
	the infimum is obtained.   Fix a manifold $M$, and let $M \stackrel{p_{n}}{\to} (S^3,L_{n})$ be 
	a sequence of covers that approximates $\lv[M]$.
	By Cao--Meyerhoff~\cite{CaoMeyehoff}, for every $n$,
	$\vol[S^{3 }\setminus L_{n}] > 2$.  Hence for 
	large enough $n$, $p_{n} \leq \lv[M] / 2$; we see that there are only finitely many 
	values for $p_{n}$.  For any collection of covers
	$M \stackrel{d}{\to} (S^3,L_{i}')$ of fixed degree $d$, the infimum of 
	$\{d \vol[S^{3} \setminus L_{i}']\}$ is
	obtained, since the set of hyperbolic volumes is well-ordered.
	It follows that the link volume is realized by some cover in  $\{M \stackrel{p_{n}}{\to} (S^3,L_{n})\}$.
	\medskip
\item[The link volume is the volume of a link exterior] that is, for any $M$, there exists $\widetilde{L} \subset M$ so that $\lv[M] = 
	\mbox{Vol}(M \setminus \widetilde{L})$.  This follows easily from the previous
	point.  Let $M \stackrel{p}{\to} (S^3,L)$ be a cover realizing the
	link volume. Let $\wt L$ be the preimage of $L$.  Then the cover $M \to S^{3}$ induces a 
	cover $M \setminus \wt L \to S^{3} \setminus L$.  Since the cover $M\setminus \wt L \to S^{3} \setminus L$ 
	is not branched, we can
	lift the hyperbolic structure on $S^{3} \setminus L$ to $M \setminus \wt{L}$.  
	We obtain a complete finite volume hyperbolic structure on $M \setminus \wt L$
	of volume $p \vol[S^{3} \setminus ]L = \lv[M]$.
	\medskip
\item[The link volume is bigger than the volume] If $M$ is hyperbolic then $\vol[M] < \lv[M]$: this follows immediately from 
	the previous point and the fact the volume always goes down under Dehn filling~\cite{NeumannZagier}.
	\medskip
\item [The spectrum of link volumes is well ordered] it follows from the second point above that the spectrum of
	link volumes is a subset of the spectrum of hyperbolic volumes.
	Since the spectrum of hyperbolic volumes is well ordered, so are all of
	its subsets.
	\medskip
\item[The spectrum of link volumes is "small"] the reader can easily make sense of the claim the the spectrum of link volumes is 
	a very small subset of the spectrum of hyperbolic volume.  In fact, the spectrum of link
	volumes is a subset of the spectrum integral products of volumes of hyperbolic links in $S^{3}$.  However,
	it is not too small: there are infinitely many manifolds $M$ with $\lv[M] < 7.22\dots$.  Moreover,
	in \cite{JairYoav} Jair Remigio-Juarez and the first named author showed that there are infinitely
	many manifold of {\it the same} link volume, just under $7.22\dots$.  This is in sharp contrast to
	the hyperbolic volume function which is finite-to-one.
	\end{description}

For the remainder of this paper we will often use these facts without reference.

\medskip

Basic questions about the Link Volume include:

\smallskip

\begin{enumerate}
\item  Calculate $\lv[M]$.  It is not clear whether or not there exists an algorithm to
	calculate the link volume of a given manifold $M$.  This would involves some 
	questions about the set of links in $S^{3}$ that give rise to $M$ and appears 
	to be quite hard. 	
\item The following question was proposed by Hitoshi Murakami: if $N \stackrel{q}{\to} M$ 
is an unbranched cover then
  $\lv[N] \leq q \lv[M]$.  How good is this bound?  Even for $q=2$, the answer is not clear.	
\item Since the link volume is obtained, for every manifold $M$ there is a positive integer $d$
which is the smallest integer so that there exists a cover
$M \stackrel{d} \to (S^3,L)$  realizing $\lv[M]$.  What is $d$ and how does
it reflect the topology of $M$?  Can $d$ be arbitrarily large?  Is any positive integer
$d$ for some $M$?
\item Characterize the set $\{ \widetilde{L} \subset M| \exists M \to
  S^3$, branched over $L$, and $\widetilde{L}$ is the preimage of $L\}$.
  The link volume is, of course, the minimal volume of the manifolds in this set,
  and in this paper we concentrate on it.
  It is easy to see that there is no upper bound to the volumes of manifolds in this set.
  It may be interesting to try and characterize the elements of this set.  
\item Do there exist hyperbolic manifolds $M_1$, $M_2$ with $\vol[M_1] =
  \vol[M_2]$ and $\lv[M_1] \neq \lv[M_2]$? 
\item Similarly, do there exist hyperbolic manifolds $M_1$, $M_2$ with $\lv[M_1] =
  \lv[M_2]$ and $\vol[M_1] \neq \vol[M_2]$? We note that the examples of manifolds with the same
  link volume mentioned above are all Siefert fibered spaces.
\end{enumerate}

\section{Proof of Theorem~\ref{thm:jt}}
\label{sec:jt}

Fix $V>0$.  Fix $\mu>0$ a Margulis constant for $\mathbb{H}^{3}$ and $d>0$.  
(We remark that the constant $\Lambda$ that we obtain in this proof depends on these choices.)

Let $M$ be a manifold of $\lv[M] < V$.  Let $M \stackrel{p} \to (S^3,L)$ be a cover
realizing $\lv[M]$.
Denote the $d$ neighborhood of the $\mu$-thick part of $S^{3} \setminus L$
by $E_{L}$.  By construction, $E_{L}$ is obtained from
$S^{3} \setminus L$ by drilling out certain geodesics; by 
Kojima~\cite[Proposition~4]{kojima}, $E_{L}$ is hyperbolic.

Let $X_{\phi}$ denote the preimage of $E_{L}$ in $M$.
Then the cover $\phi:M \to S^{3}$ induces an unbranched
cover $\phi:X_{\phi} \to E_{L}$.  By lifting the hyperbolic 
structure from $E_{L}$ to $X_{\phi}$, we see that
 $X_{\phi}$ is a finite volume hyperbolic manifold.

By construction, the following diagram commutes (where vertical arrows
represent the covering projections and horizontal arrows represent
Dehn fillings):

\begin{center}
\begin{picture}(200,60)(0,0)
  \put(  0,  0){\makebox(0,0){$E_{L}$}}
  \put(  0,50){\makebox(0,0){$X_{\phi}$}}
  \put(  0, 40){\vector(0,-1){30}}   
  \put( 10,  0){\vector(1,0){75.5}}
  \put(100,  0){\makebox(0,0){$S^3,L$}}
  \put(100, 40){\vector(0,-1){30}}
  \put(8,28){\makebox(0,0){$/\phi$}}
  \put(108,28){\makebox(0,0){$/\phi$}}
  \put(100,50){\makebox(0,0){$M$}}
  \put( 10,50){\vector(1,0){80}}
\end{picture}
\end{center}

\bigskip

By J\o rgensen and Thurston (see, for example, \cite{KR}), there exists a constant
$\Lambda$ (depending on $\mu$ and $d$), so that for any complete, finite volume hyperbolic manifold $N$,
the $d$-neighborhood of the $\mu$-thick part of $N$ can be triangulated
using no more than $\Lambda \vol[N]$ tetrahedra.  Applying this to 
$N = S^{3} \setminus L$, since the $d$-neighborhood of the $\mu$-thick
part of $N$ is $E_{L}$, we see that $E_{L}$ can be triangulated using at most 
$\Lambda \vol[S^{3} \setminus L] = \Lambda \lv[M] / p < \Lambda V/p$ tetrahedra.

Since there are only finite many manifolds that can be triangulated using
at most $\Lambda V/p$ tetrahedra, there are only finitely many possibilities 
for $E_{L}$.

Lifting the triangulation from $E_{L}$ to $X_{\phi}$, we see that $X_{\phi}$
can be triangulated with at most $\Lambda \lv[M] < \Lambda V$ tetrahedra, and that 
$\phi:X_{\phi} \to E_{L}$ is simplicial.  This shows that there are only finitely many
possibilities for $X_{\phi}$ and $\phi$.  We denote them $\{\phi_{i}:X_{i} \to E_{i}\}_{i=1}^{n_{V}}$.

\section{The Link Volume and Dehn Filling}
\label{sec:dehn}

In this section we prove Theorem~\ref{thm:dehn}.  The proof is constructive and requires two 
elements, the first is Hilden's
construction of simple 3-fold covers of $S^{3}$, and the second is the results of Thurston and 
Menasco that show that an alternating link that ``looks like'' a hyperbolic link is in fact hyperbolic.  
For the latter, see Subsection~\ref{sebsection:alternating}.  We now explain the former.

In~\cite{hilden}, Hilden showed that any 3-manifold is the simple 3-fold cover of $S^{3}$.  
The crux of his proof is the construction,
for any $g$, of a 3-fold branched cover $p:V_{g} \to B$, where $V_{g}$ is the genus 
$g$ handlebody and $B$ is the 3-ball.  He then proves that any map 
$f:\partial V_{g} \to \partial V_{g}$ can be isotoped so as to commute with $p$.  
Thus $f$ induces a map $\bar f:\partial B \to \partial B$ so that the following diagram 
commutes (here the vertical arrows denote Hilden's covering projection):

\begin{center}
\begin{picture}(200,60)(0,0)
  \put(  0,  0){\makebox(0,0){$B$}}
  \put(  0,50){\makebox(0,0){$V_{g}$}}
  \put(  0, 40){\vector(0,-1){30}}   
  \put( 10,  0){\vector(1,0){75.5}}
  \put(100,  0){\makebox(0,0){$B$}}
  \put(100, 40){\vector(0,-1){30}}
  \put(50,10){\makebox(0,0){$\bar f$}}
  \put(50,60){\makebox(0,0){$f$}}
  \put(100,50){\makebox(0,0){$V_g$}}
  \put( 10,50){\vector(1,0){80}}
\end{picture}
\end{center}

\bigskip

Starting with a closed, orientable, connected 3-manifold $M$, Hildens uses a Heegaard splitting of
$M = V_g \cup_f V_g$; the construction above gives a map to $B \cup_{\bar f} B \cong S^3$.  
This is, in a nutshell, Hilden's construction of $M$ as a cover of $S^3$.

Our goal is using a similar construction to get a map from $X$.  Since $X$ has boundary it cannot 
branch cover $S^{3}$, and we must modify Hilden's
construction.  To that end, we first describe the cover $p:V_{g} \to B$ in detail.  
Let $S_{3g+2}$ be the $3g+2$ times punctured $S^2$, viewed as a $3g$-times punctured annulus.
Then $S_{3g+2} \times [-1,1]$ admits a symmetry of order two (rotation by $\pi$ about 
the $y$-axis) given by $(x,y,t) \mapsto (-x,y,-t)$, 
where $S_{3g+2}$ is embedded symmetrically in the $xy$-plane as shown in 
Figure~\ref{fig:annulus}.
\begin{figure}
\includegraphics{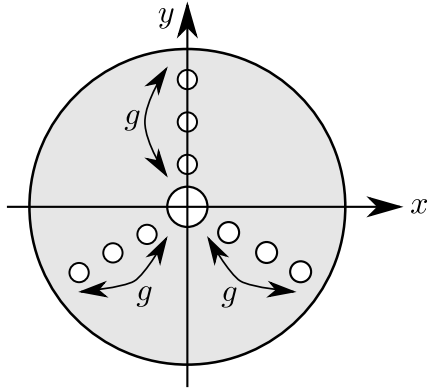}
\caption{$S_{3g+2}$ embedded in the $xy$-plane}
\label{fig:annulus}
\end{figure}

$S_{3g+2} \times [-1,1]$ 
also admits a symmetry of order 3 by rotating $S_{3g+2}$ about the origin of the 
$xy$-plane and fixing the $[-1,1]$ factor.  These two symmetries generate an action
of the dihedral group of order 6 on $S_{3g+2} \times [-1,1]$.  It is easy to see that the 
quotient is a ball.  On the other hand, the quotient of  $S_{3g+2} \times [-1,1]$ 
by the order two symmetry is $V_{g}$.  This induces the map $f:V_{g} \to B$; note that this
is a cover, branched along a trivial tangle with $g+2$ arcs (thus the branch set of
the map $M \to S^3$ 
described above is a $g+2$ bridge link, and the braiding is determined by $\bar f$).  This is 
Hilden's construction, see Figure~\ref{fig:annulus1}, where the branch set of $V_{g} \to B$
is indicated by dashed lines (in $B$).
\begin{figure}
\includegraphics{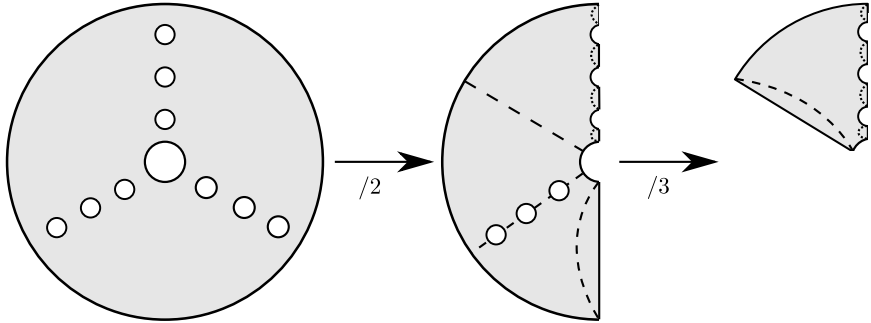}
\caption{Hilden's covers}
\label{fig:annulus1}
\end{figure}

A Heegaard splitting for the manifold with boundary $X$ is
a decomposition of $X$ into two compression bodies; we assume the reader is familiar 
with the basic definitions (see, for example, \cite{CassonGordon}).   We use the notation $V_{g,n}$
for a compression body with $\partial_{+} V_{g,n}$ a genus $g$ surface and 
$\partial_{-} V_{g,n}$ a collection of $n$ tori (so $0 \leq n \leq g$).   Since $\partial X$ consists
of $n$ tori, any Heegaard splitting of $X$ consists of two compression bodies of the form $V_{g,n_{1}}$ 
and $V_{g,n_{2}}$, for some $g,n_1,n_2$ with $n_{1} + n_{2} = n$.   We use the notation
$V_{g,n_{i}}^{*}$ for the manifold obtained by removing $n_{i}$ disjoint open balls from the interior of $V_{g,n_{i}}$.  
We use the notation $X^{*}$ for the manifold obtained by removing $n$ disjoint open balls from the 
interior of $X$.  Finally, we use the notation 
$B^{*}_{n_{i}}$ for the manifold obtained by removing $n_{i}$ disjoint open balls from the interior of $B$.  

Since compression bodies do not admit simple 3-fold branched covers of the type we need, we work 
with $V_{g,n_{i}}^{*}$, see Figure~\ref{fig:annulus2}. 
\begin{figure}
\includegraphics{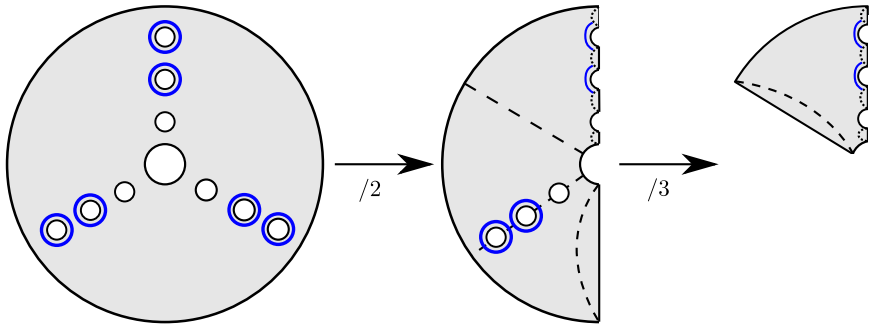}
\caption{Hilden's covers modified}
\label{fig:annulus2}
\end{figure}
Figure~\ref{fig:annulus2} is very similar to Figure~\ref{fig:annulus1}, 
but has a few ``decoration'' added in blue.
The circles added to $S_{3g+2} \times [-1,1]$ are embedded in 
$S_{3g+2} \times \{0\}$.  There are exactly $3n_{i}$ such circles.
Clearly, they are invariant under the dihedral group action, and their images in $V_{g}$ and $B$ are 
shown.  By removing an appropriate neighborhood of these circles and their images,
we get a simple 3-fold cover from $V_{g,n_{i}}^{*}$ to $B_{n_{i}}^{*}$.  

Applying Hilden's theorem to the gluing map 
$f:\partial_{+} V_{g,n_{1}} \to \partial_{+} V_{g,n_{2}}$, we obtained a map $\bar f:\partial B_{n_{1}} \to \partial B_{n_{2}}$.  
Clearly, downstairs we see the manifold obtained by removing $n_{1} + n_{2} = n$ open balls from $S^{3}$; 
we denote it by $S^{3,*}$.

Note that the branch set is a tangle (that is, a 1-manifold properly embedded in $S^{3,*}$) that intersects every 
sphere boundary component in exactly 4 points; we denote this branch set by $T$.  
Moreover, the preimage of each component of $\partial S^{3,*}$ 
consists of exactly two components: a torus that double covers it, and a sphere that projects to it homeomorphically.  
The map from the torus in $\partial X^{*}$ to the sphere in $S^{3,*}$ is the quotient under the well known hyperelliptic involution.

\bigskip

\noindent Hilden's construction, as adopted to our scenario, is the key to everything we do below.  We
sum up its main properties here:

\begin{prop}
\label{prop:hilden}
Let $X$ be a compact, orientable manifold with $\partial X$ consisting of $n$ tori.  Let $X^{*}$ be the manifold 
obtained  by removing $n$ open balls from the interior of $X$.  
Let $S^{3,*}$ be the manifold obtained by removing $n$ open balls from $S^{3}$.

Then there exists a simple 3-fold cover $p:X^{*} \to S^{3,*}$.  The branch set is a compact 1-manifold, denoted $T$, 
that intersects every boundary component of $S^{3,*}$
in exactly four points.  

The preimage of each component $S$ of $\partial S^{3,*}$ consists of one torus component of $\partial X$ that 
double covers $S$ via a hyperelliptic involution, and one sphere component of $\partial X^{*} \setminus \partial X$
that maps to $S$ homeomorphicaly.
\end{prop}

Recall that in Theorem~\ref{thm:dehn}, $X$ came equipped with a choice of meridian and longitude on each 
boundary component.  $S^{3,*}$ is naturally a subset of $S^{3}$.
We isotope $S^{3,*}$ in $S^{3}$ so that, after projecting it into the plane, the following 
conditions hold:
	\begin{enumerate}
	\item The balls removed from $S^{3}$ are denoted $\bar B_{i}$ ($i=1\dots,n$).  The projection of each $\bar B_{i}$
	is a round disk; these disks are denoted $B_i$, see Figure~\ref{fig:Bi}.
	\begin{figure}
	\includegraphics{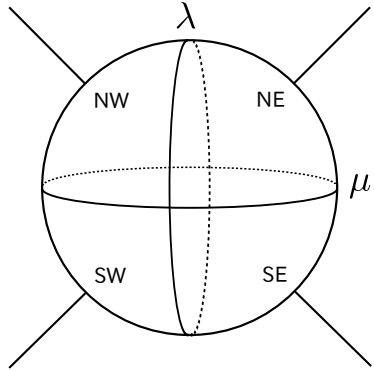}
	\caption{$T$ in a neiborhood of $B_{i}$}
	\label{fig:Bi}
	\end{figure}
	\item $T$ intersects each $B_i$ in exactly four points.  Each of these point is an endpoint of a  strand of $T$.  The four point are the intersection
	of the lines of slopes $\pm 1$ through the center of the disk with its boundary, and are labeled (in cyclic order) NE, SE, SW, and NW.
	\item We twist the boundary components of $S^{3,*}$ so that, in addition, the meridian and longitude of the 
	corresponding boundary component of $\partial X$ map to a horizontal and vertical circles, respectively;
	thses curves (slightly rounded) are labeled $\mu$ and $\lambda$ in Figure~\ref{fig:Bi}. 
	\end{enumerate}

Let $T_i \subset \partial X$ be the torus that projects to $\partial \bar B_i$.  Recall that by {\it Dehn filling} $T_i$ 
we mean attaching a solid torus $V$ to $T_i$.   $V$ is foliated by concentric tori, with one singular leaf (the core circle).  
To understand how the hyperelliptic involution extends from $\partial V = T_{i}$ into $V$ we construct the following
explicit model of the hyperelliptic involution:  let $T_{i}$ be the image of $\mathbb R^{2}$ under the action
of $\mathbb Z^{2}$ given by $(x,y) \mapsto (x+n,y+m)$.  Then the hyperelliptic involution is given by
rotation by $\pi$ about $(0,0)$.  The four fixed points on $T_{i}$ 
are the images of $(0,0)$, $(1/2,0)$ (rotate and translate by
$(x+1,y)$), $(0,1/2)$ (rotate and translate by $(x,y+1)$), and $(1/2,1/2)$ (rotate and translate by
$(x+1,y+1)$).  Given any slope $p/q$ (with $p$ and $q$ relatively prime), 
it is clear that the foliation of $\mathbb R^{2}$ by 
straight lines of slope $p/q$ is
invariant under the rotation by $\pi$ about $(0,0)$.  The line through $(0,0)$ goes through $(p/2,q/2)$,
which is the image of one of the other three fixed points, as not both $p$ and $q$ are even.  Similarly
for the lines through $(1/2,0)$, $(0,1/2)$, $(1/2,1/2)$; these lines project to two circles on the torus, with
exactly two fixed points on each circle.  By considering the images of the foliation of $\mathbb R^{2}$ by
lines of slope $p/q$, we obtain a foliation of $\partial V$ by circles (each representing the slope $p/q$).
In the foliation of $V$ by concentric tori, each torus admits such a foliation, and the length of the leaves
limit on $0$ as we approach the singular leaf.  At the limit, we see that the involution on the non-singular leaves 
induces an involution of the singular leaf whose image is an arc.  Thus the hyperelliptic involution of $T_{i}$
extends to an involution on $V$, whose image is foliated by spheres, with one singular leaf that is an arc.
The image of $V$ is a ball, and the branch set is a rational tangle of slope $p/q$;
for more about rational tangles and their double covers see, for example,~\cite{rolfsen}.
We denote this rational tangle by $R_{i}$.

\begin{notation}	
\label{notation:EquivalenceClasses}
We assume the rational tangles we study have been isotoped to be alternating (it is well known that this can be achieved).  
Two rational tangles are considered
{\it equivalent} if the following two conditions hold:
	\begin{enumerate} 
	\item  The over/under information of the strands of the rational tangles coming in from the NE are the same.
	Since the rational tangle is alternating, this implies that the over/under information from the other corners is the same as well.
	\item The strands of the rational tangles that start at NE end at the same point (SE, SW, or NW).
	\end{enumerate} 
Note that the crossing information is ill-defined for the two tangles $1/0$ and $0/1$, as they have no crossings.  
We arbitrarily choose an equivalent class for each of these tangle, so that the second condition is fulfilled.
We obtain $6^n$ possible equivalence classes (recall that $n = {|\partial X|}$). 
\end{notation}

\medskip

\noindent Given slopes on $T_{1},\dots,T_{n}$, we get rational 
tangles $R_{1},\dots,R_{n}$, as described above.
In each $\bar B_{i}$ we place a rational tangle, denoted $\widehat R_{i}$,
so that $\widehat R_{i} \in \{\pm 1, \pm 2, \pm 1/2 \}$, 
representing the same equivalence class as $R_{i}$.  
We assume that their projections into $B_{i}$ are as in Figure~\ref{fig:onetwohalf}.
We thus obtain a link, denoted $\widehat T$, and a diagram for $\widehat T$,
denoted $\widehat D$.  Since $\widehat T$ and $\widehat D$ only depend on the
equivalence classes of the slopes, when considering all possible slopes,
we obtain finitely many links and diagrams (specifically, $6^{n}$). 

In order to obtain hyperbolic branch set, we will, eventually, apply Mensaco~\cite{menasco} 
as explained in Subsection~\ref{sebsection:alternating}.
To that end we will need to make the branch set alternating.  
As we shall see below, we do this using a $1 \to -2$ and $-2 \to 1$ Montesinos
moves; these moves can be used to make the link alternating in a way that
is very similar to crossing changes.  Below we will show that we can apply
Montesinos moves to $T$, however, we may not apply these moves to the rational
tangles inside $B_{i}$.  This causes the following trouble: let $\alpha \subset T$ be an interval 
connecting two punctures, say $\partial \bar B_{i}$ and $\partial \bar B_{i'}$ (possibly, $i=i'$). 
Assume that the last crossing of $R_{i}$ before $\alpha$ is an over crossing, and  
that the number of crossings along $\alpha$ is even.  
Then if we make $T$ alternate, the last crossing along $\alpha$ will be an overcrossing.
This means that the first crossing of $R_{i'}$ after $\alpha$ must be an undercrossing.  
This may or may not be the case, and we have no control over it.

In order to encode this, we consider the following graph $\Gamma$:  $\Gamma$ has $n$ vertices, and they
correspond to $B_{1},\dots,B_{n}$.  The edges of $\Gamma$ correspond to intervals of $T$ 
that connect $B_{i}$ to $B_{i'}$ (again, $i$ and $i'$ may not be distinct).
Inspired by the discussion above, we assign signs to the edges of $\Gamma$ as follows (in essence,
good edges get a $+$ and bad edges get a $-$):
	\begin{enumerate}
	\label{SignsOfEdges}
	\item Let $I \subset T$ be an interval 
	connecting $B_{i}$ to $B_{i'}$ (possibly $i=i'$)
	so that the last
	crossing before $I$ and the first crossing after $I$ are the same (that is,
	both overscrossings or are both undercrossings),
	and the number of crossings along $I$ is odd.  Then the corresponding 
	edge get the sign $+$.
	\item Let $I \subset T$ be an interval 
		connecting $B_{i}$ to $B_{i'}$ (possibly $i=i'$)
		so that the last
	crossing before $I$ and the first crossing after $I$ are the opposite (that is,
	one is an overcrossing and one an undercrossing),
	and the number of crossings along $I$ is even.  Then the corresponding 
	edge get the sign $+$.
	\item All other edges get the sign $-$.
	\end{enumerate}

If $\Gamma$ is connected, we pick a spanning tree $\widehat \Gamma$ for $\Gamma$.  That is, $\widehat \Gamma$ is a 
tree obtained from $\Gamma$ by removing edges, so that every vertex of $\Gamma$ is adjacent to some edge of 
$\widehat \Gamma$.  In general, we take $\widehat \Gamma$ to be 
a {\it maximal forest} in $\Gamma$.  A forrest is a collection of trees, or a graph without cycles.  
A {\it maximal forest} in $\Gamma$ is a graph obtained
from $\Gamma$ by removing a minimal (with respect to inclusion) set of edges so that a forrest is obtained; equivalently,
it is the union of maximal trees for the connected components of $\Gamma$.  Clearly a 
maximal forest $\widehat\Gamma$ has the following two properties: first,
$\widehat\Gamma$ contains no cycles.  
Second, any edge from $\Gamma$ that we add to $\widehat \Gamma$ closes a cycle.

\begin{lem}
\label{lemma:VertexSigns}
There is a sign assignment to the vertices of $\widehat \Gamma$, so that an edge of 
$\widehat \Gamma$ has a plus sign if and only if the vertices it connects
have the same sign.
\end{lem}

\begin{proof}[Proof of Lemma~\ref{lemma:VertexSigns}]
We induct on the number of edges in $\widehat \Gamma$.  If there are no edges there is nothing to prove.

Assume there are edges.  In that case at least one component of $\widehat \Gamma$ is a tree with more that one vertex.  
Such a tree must have a leaf, say $v$.  
Remove $v$ and $e$, the unique edge of $\widehat \Gamma$ connected to $v$.  
By induction, there is a sign assignment for the remaining 
vertices fulfilling the conditions of the lemma.  We now add $v$ and $e$.  Clearly, we can give $v$ a 
sign so that the condition of the lemma holds 
for $e$.  The lemma follows.
\end{proof}

We now isotope $\widehat T$ and accordingly modify 
$\widehat D$ as shown in Figure~\ref{that connectfig:YamashitaMove}
\begin{figure}
\includegraphics{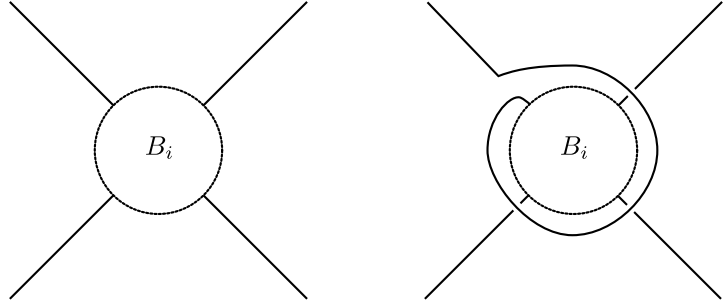}
\label{that connectfig:YamashitaMove}
\caption{Modifying $\widehat D$ near $B_{i}$ (negative sign)}
\end{figure}
at each puncture that corresponds to a vertex with a minus sign.  
Since this changes the number of crossings on some of strands of $\Gamma$, we recalculate the signs on the corresponding
edges.  Note that the isotopy above adds one or three
crossing to every strand of $T$ that corresponds to an edge of $\widehat \Gamma$ with sign $-$, 
and zero, two, four, or six crossings to 
every strand of $T$ that corresponds to an edge with sign $+$.  We easily conclude that
every edge of $\widehat \Gamma$ has sign $+$.  Moreover:

\begin{lem}
\label{lemma:SingsOfEdges}
Every edge of $\Gamma$ has sign $+$.
\end{lem}

\begin{proof}
The proof is very similar to the proof that every link projection can be made into an alternating link projection via crossing change and is 
left for the reader, with the following hint: suppose there exists an edge, say $e$,
whose sign is $-$.  Since we used a maximal forest, there is a 
cycle in $\Gamma$ (say $e_{1},\dots,e_{k}$) so that $e_{1} = e$ and $e_{i}$ belongs to the maximal forest for $i>1$; in particular,
exactly one edge of the cycle has sign $-$.  Use this cycle to produce a closed curve (not necessarily simple)
in $S^{2}$ that intersects he link $\widehat T$ transversely an odd number of times.  This is absurd, in light of the Jordan Curve Theorem.
\end{proof}

Next we prove that $X^{*}$ can be obtained as a 3-fold cover of $S^{3,*}$
with a particularly nice branch set. 
We begin with $T$, $\widehat T$, and $\widehat D$ described above; their properties are
summed up in Condition~(1) of Lemma~\ref{lem:hyperbolicity} below.
For parts of this argument {\it cf.} Blair~\cite{blair}.  Recall Definitions~\ref{dfn:alternating} 
for standard terms in knot theory.

\begin{lem}
\label{lem:hyperbolicity}
There exists a link $\widehat T$ in $S^{3}$, with projection into $S^{2}$ denoted $\widehat D$, 
so that $X^{*}$ is a simple, $3$-fold cover of 
$S^{3} \setminus \cup_{i=1}^{n} \mbox{\rm int}(\bar B_{i})$ branched along the tangle 
$T = \widehat T \cap (S^{3} \setminus \cup_{i} \mbox{\rm int}(B_{i}))$ 
and the following conditions hold:  
	\begin{enumerate}
	\item $\widehat T \cap \bar B_{i} = \widehat R_{i}$ (recall that $\widehat R_{i}$ projects into 
	$B_{i}$ as shown in Figure~\ref{fig:onetwohalf}), 
	\begin{figure}
	\includegraphics{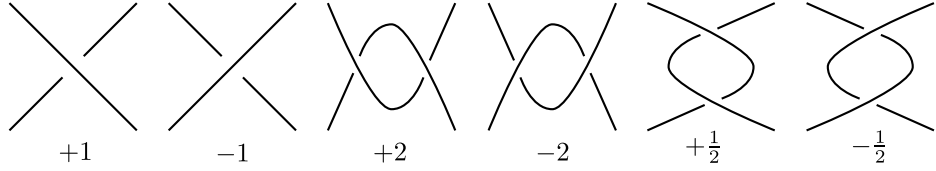}
	\caption{$\hat R_{i}$, the representatives for equivalence classes}
	\label{fig:onetwohalf}
	\end{figure}
	the projection of $\mbox{int}(T \setminus \widehat T)$ is disjoint from $\cup_{i=1}^{n}B_{i}$,
	and the meridian and longitude of $T_{i} \subset \partial X$ project to horizontal and vertical circles about $B_{i}$
	(respectively, recall Figure~\ref{fig:Bi}).
	\item $\widehat D$ is not a split diagram.
	\item Every simple closed curve in $S^{2} \setminus \cup B_{i}$ that intersects $\widehat D$
	transversely in two simple point bounds a disk that intersects $\widehat D$ in a single arc with no
	crossing. 
	\item Let $\alpha \subset S^{2}$ be an arc with one endpoint on $B_{i'}$  and the other on $B_{i''}$ 
	(for $i', \ i'' =1,\dots,n$, possibly $i'=i''$), and $\mbox{\rm int}(\alpha) \cap (\cup_{i} B_{i}) = \emptyset$.
	Then one of the following conditions holds:
		\begin{enumerate}
		\item $i' = i''$, and $\alpha$ cobounds a disk with $\partial B_{i}$ with
		no crossings.
		\item $|\mbox{int}(\alpha) \cap \widehat D| > 2$.
		\end{enumerate} 
	\item In the three coloring of $\widehat D \cap (S^{2} \setminus \cup_{i=1}^{n} B_{i})$ induced by the 
	cover $X^{*} \to S^{3,*}$, every crossing is three colored.
	\item $\widehat D$ is alternating.
	\item $\widehat T$ is a knot.
	\end{enumerate}
\end{lem}

\begin{rmk}
\label{rmk:signs}
To obtain conditions~(1)--(5) we modify $\widehat T$ via isotopy; except for the move
shown in Figure~\ref{fig:localmove2}, 
the projection of the support of this isotopy is disjoint from $\cup_{i=1}^{n}B_{i}$.  
Note that in the move shown in Figure~\ref{fig:localmove2} each edge gets and even number
of crossings added.
Hence the signs of the edges of $\Gamma$ do not change,
and Lemma~\ref{lemma:SingsOfEdges} still holds after we obtain  
Conditions~(1)--(5).
(We use this lemma to obtain Condition~(6), and
never need it again after that.)  
\end{rmk}

\begin{proof}
{\bf Condition~(1).} Condition~(1) already holds.  We note that none of the moves applied
in the proof of this lemma changes this.  We will not refer to Condition~(1) explicitly.

\bigskip

\noindent {\bf Condition~(2).}  $\widehat D$ is diagrammatically split if and only if it is disconnected.
Suppose $\widehat D$ is disconnected, and let $K_j$ and $K_{j'}$ be components of $\widehat T$ that project
to distinct components of $\widehat D$.  
Let $\alpha \subset S^{2} \setminus \cup_{i=1}^{n} B_{i}$ be an embedded arc
with one endpoint on $K_{j}$ and the other on $K_{j'}$ (note that $K_{j}, \ K_{j'} \not\subset \cup_{i=1}^{n} B_{i}$, hence $\alpha$ exists;
$\alpha$ may intersect $\widehat D$ in its interior).  We perform an
{\it isotopy along $\alpha$}, as shown in Figure~\ref{fig:alphamove}.
\begin{figure}
\includegraphics{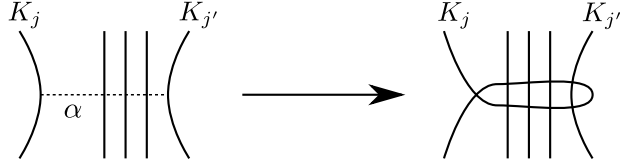}
\caption{Isotopy along $\alpha$}
\label{fig:alphamove}
\end{figure} 
After that $K_j$ crosses $K_{j'}$ outside $\cup_{i=1}^n B_i$; clearly, this reduces the number of components
of $\widehat D$.
Repeating this process if necessary, Condition~(2) is obtained.

\bigskip

\noindent  {\bf Condition~(3).}   For each $B_i$, let $N(B_i)$ be a normal neighborhood of $B_i$ so that
$\widehat D \cap N(B_i)$ consists of the tangle in $B_i$ and four short segments as in the left hand side of 
Figure~\ref{fig:localmove2}.  We assume further that for $i \neq j$, $N(B_i) \cap N(B_j) = \emptyset$.
\begin{figure}
\includegraphics{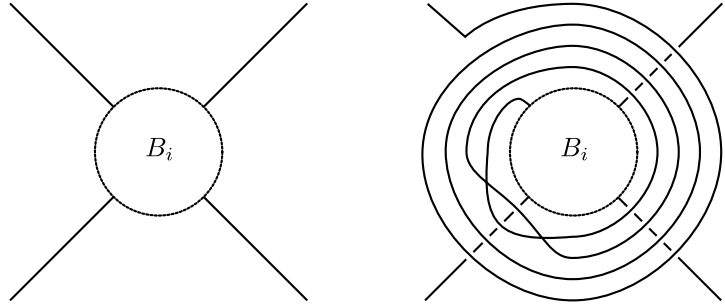}
\caption{Isolating $B_i$}
\label{fig:localmove2}
\end{figure}
Inside each $N(B_i)$ perform the isotopy shown in Figure~\ref{fig:localmove2}. 

Next we count the number of simple closed curves in $S^{2} \setminus \cup_{i} B_{i}$ that intersect $\widehat D$
in two points and do not bound a disk $\Delta$ with $\widehat D \cap \Delta$ is a single arc with no crossings.
These curves are counted up to ``diagrammatic isotopy'', that is, an isotopy via curves that are 
transverse to $\widehat D$ at all time and in  particular are disjoint from the crossings.
 
Let $C_1,\dots,C_k$ be the closures of the components of $S^2 \setminus (\widehat D \cup (\cup_i B_i))$.
Let $\gamma, \gamma' \subset S^2 \setminus \cup_i B_i$ be two simple closed curves that intersects 
$\widehat D$ transversely in two simple points.  Then $\widehat D$ cuts $\gamma$ into 2 arcs, say one in the region 
$C_j$ and one in $C_{j'}$.  Note that if $j=j'$, then $C_j$ is adjacent to
itself, and in particular there is a simple closed 
curve in $S^2$ that intersects $\widehat D$ transversely in one point, which is absurd.  
Condition~(2) (connectivity of $\widehat D$)
is equivalent to all regions being disks, and hence implies
that $\gamma$ and $\gamma'$ are diagrammatically isotopic if and only if both curves traverse  
the same regions $C_j$ and $C_{j'}$, and $\gamma \cap \partial C_j$ is contained in the same segments of 
$C_j \cap C_{j'}$ as $\gamma' \cap \partial C_{j'}$.  (See Figure~\ref{fig:regions};
\begin{figure}
\includegraphics{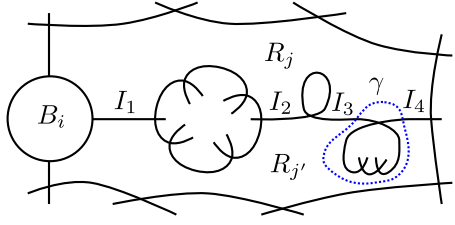}
\caption{Segments}
\label{fig:regions}
\end{figure}
here a {\it segment} means an interval $I \subset S^{2} \setminus \cup_{i} \mbox{int}(B_{i})$, so that
$I \subset C_j \cap C_{j'}$, $\partial I$ are crossings or lie on $\partial B_i$ for some $i$, and  
$I$ contains no crossings in its interior.) 
For any pair of regions $C_j$ and $C_{j'}$,  let $n_{j,j'}$ be the number of segments in $C_j \cap C_{j'}$ 
(for example, in Figure~\ref{fig:regions}, $n_{j,j'} = 4$).  
Then we see that the number
of simple closed curves that intersect $\widehat D$ in two simple points, 
traverse $C_j$ and $C_{j'}$, and do not bound a disk
containing a single arcs of $\widehat D$ (counted up to diagrammatic isotopy) 
is ${n_{j,j'}}\choose{2}$, where 
${0}\choose{2}$ and ${1}\choose{2}$ are naturally understood to be 0.
Hence the total number of such curves (counted up to diagrammatic isotopy) is 
\begin{equation}
\label{equation:NumberOfCircles}
\sum_{1 \leq j < j' \leq k}{n_{j,j'}\choose2}.
\end{equation}

Now assume that condition~(3) does not hold; then there exist regions 
$C_j$ and $C_{j'}$ with $n_{j,j'} \geq 2$.  Let $I$ be
an interval of $C_j \cap C_{j'}$.  Since we isolated $B_i$ (for all $i$) as 
shown in Figure~\ref{fig:localmove2}, the endpoints of $I$ cannot lie on 
$\partial B_i$ and must therefore both be crossings. The move shown in 
Figure~\ref{fig:localmove3} 
\begin{figure}
\includegraphics{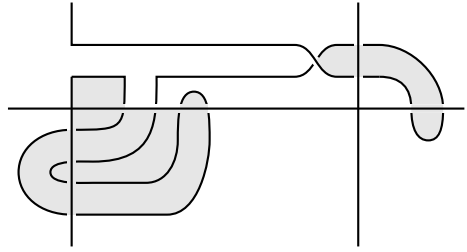}
\caption{Isotopy to reduce $C_{i,i'}$}
\label{fig:localmove3}
\end{figure}
reduces $n_{j,j'}$ by one.  
This move introduces several new regions, and those are shaded in 
Figure~\ref{fig:localmove3}. 
Inspecting Figure~\ref{fig:localmove3}, we see that
for any pair of regions $C_{j}$, $C_{j'}$ that existed prior to the move,
$n_{j,j'}$ does not increase, and for any pair
of regions $C_{j}$, $C_{j'}$  with at least one new region, $n_{j,j'}$
is 1 or 0.  Hence the sum in Equation~(\ref{equation:NumberOfCircles}) 
is reduced, and repeated application of this move 
yields a diagram $\widehat D$ for a link $\widehat T$ for which Condition~(3)
holds; by construction, Condition~(2) still holds.

\bigskip

\noindent  {\bf Condition~(4).}  Condition~(4) holds thanks to the isotopy performed in the 
previous step and shown in Figure~\ref{fig:localmove2}.

\bigskip

\noindent  {\bf Condition~(5).}  Since $\widehat D$ is the branch set of the simple 3-fold
cover $X^{*} \to S^{3,*}$ it inherits a 3-coloring as explained in Subsection~\ref{subsec:covers},
where the colors are transpositions in $S_{3}$.  Since $X^{*}$ is connected, at least
two colors appear in the coloring of $T$ (recall Lemma~\ref{lem:connected3fold}; that
lemma was stated for covers of $S^{3}$ but it is easy to see that it holds for covers of
$S^{3,*}$ as well).

Assume there exists a one colored crossing of $\widehat D$ outside $\cup_{i=1}^n B_i$, 
say $c$, and let $p$
be a point on a strand of $\widehat D$ that is of a different color than $c$, and so
that $p \not\in \cup_{i=1}^n B_i$.  Let $\alpha$ be an arc connecting $p$ and $c$
so that $\alpha \cap (\cup_{i=1}^n B_i) = \emptyset$.  If $\mbox{int}(\alpha)$ intersects 
a strand of $\widehat D$ whose color is different than the color of $c$, we cut
$\alpha$ short at that intersection.  Thus we may assume that any point
of $\mbox{int}(\alpha) \cap \widehat D$ has the same color as $c$.
We apply the following move (often used by Hilden, Montesinos and others), see Figure~\ref{fig:threecolors}.
\begin{figure}
\includegraphics{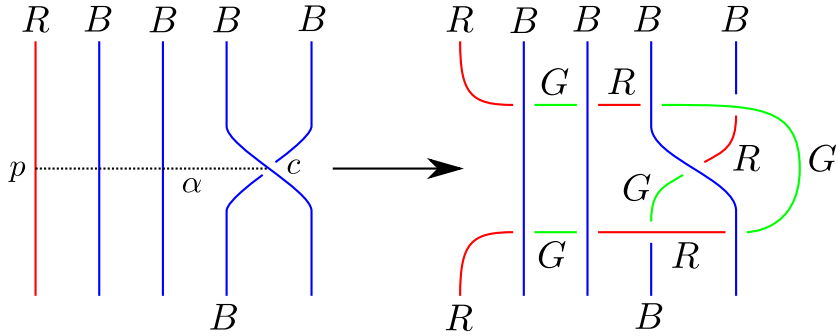}
\caption{Making the crossings 3-colored}
\label{fig:threecolors}
\end{figure}
This move reduces the number of one colored corssings outside $\cup_i B_i$,
and hence repeating this move gives Condition~(5).

We now verify that Conditions~(2)--(4) still hold.
Inspecting Figure~\ref{fig:threecolors}, we see that
Condition~(2), which is equivalent to connectivity of $\widehat D$, clearly holds.
A simple closed curves that intersects $\widehat D$ twice after this moves, intersects
it at most twice before the move.  By considering these curves and Figure~\ref{fig:threecolors}
we conclude that Condition~(3) holds as well (in checking this, note that $\mbox{int}\alpha \cap \widehat D$
maybe empty; to rule out one case, you need to use the coloring: a red arc cannot be connected
to a blue arc without a crossing).  
For each $i$, the preimage of $\partial \bar B_{i}$
is disconnected; hence the four segments of $\widehat D$ on the left side of
Figure~\ref{fig:localmove2} are all the same color.  Since $\widehat D$ is 
connected and has more than one color, is must have a three colored crossing, which
cannot be contained in $N(B_{i})$ for any $i$.  We can take the point $p$ in the construction
above to be a point near that three colored crossing, and in particular, we may assume that 
$p \not\in N(B_{i})$ for any $i$.  Therefore this move effects
$\widehat D \cap N(B_{i})$ by {\it adding} arcs that traverse $N(B_{i})$
without interscting $B_{i}$ itself, but not 
changing any of the existing diagram in the right hand side of 
Figure~\ref{fig:localmove2}.
Therefore Condition~(4) holds.

\bigskip

\noindent  {\bf Condition~(6).}  Note that the tangles $\widehat R_i$ are alternating ($i=1,\dots,n$). 
It is well known that any link projection can be made into an alternating projection by
reversing some of its crossings.  We 
mark the crossings of $\widehat D$ by $\pm$, marking a crossing $+$ if we do not need to reverse it and $-$ otherwise.
By reversing all the signs if necessary, we may assume that the signs in $B_{1}$ are $+$.  Since the signs of all the edges of
$\Gamma$ are $+$ (Lemma~\ref{lemma:SingsOfEdges} and Remark~\ref{rmk:signs}), the signs in every $B_{i}$ are all $+$.   Thus all the crossings
that are marked $-$ are outside $\cup_{i=1}^n B_{i}$, and hence three colored.  We change each of this crossing using 
the Montesinos move $+1 \mapsto -2$ or $-1 \mapsto +2$, as in the top row of Figure~\ref{fig:montesinos}, noting that this does not
change the double cover. 
It is clear that now $\widehat D$ is an alternating diagram fulfilling Conditions~(1)--(6).

\bigskip

\noindent  {\bf Condition~(7).}
Assume $T$ is a link.  If there is a crossing outside $\cup_{i} B_{i}$ that
corresponds to two distinct components of $T$, we perform a $+1 \mapsto +4$
or $-1 \mapsto -4$ Montesinos move; this reduces the number of components
of $T$.  Assume there is no such crossing, and let $\alpha$ be an arc connecting 
strands (say $s_1$ and $s_{2}$) that correspond to two distinct components of $T$.  
Since no $B_{i}$ contains a closed component, we may assume 
$\alpha \cap (\cup_{i} B_{i}) = \emptyset$; furthermore,
by truncating $\alpha$ if necessary,
we may assume that $\mbox{int}\alpha \cap \widehat D = \emptyset$.  
 By Condition~(4) at least one endpoint of
$s_{2}$ is a crossing outside $\cup_{i} B_{i}$, say $c$.
If $s_{1}$ and $s_{2}$ have 
the same color, we replace $\alpha$
with an arc that connects $s_{1}$ with a strand adjacent to $s_{2}$ at $c$.
By Condition~(5) $c$ is three colored, and by assumption, both its strands correspond to the 
same component of $T$.  Thus we obtain
an arc that connects distinct components and has endpoints of different colors.
Finally, we assume without loss of generality that the crossing information at $s_{1}$ is as shown
in Figure~\ref{fig:makingknot}.  Since $\widehat D$ is connected
and alternating, considering the face containing $\alpha$, we conclude
that the crossing information on $s_{2}$ is as shown in that figure.
We change $\widehat D$ using a $0 \mapsto \pm3$ Montasinos move
(as shown in the bottom of Figure~\ref{fig:montesinos}),
obtaining a diagram fulfilling Conditions~(1)--(6)
that corresponds to a link with fewer components, see Figure~\ref{fig:makingknot}.
\begin{figure}
\includegraphics{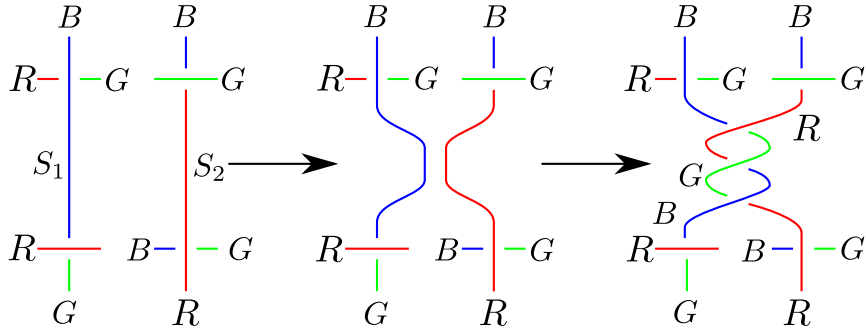}
\caption{Making the brach set into a knot}
\label{fig:makingknot}
\end{figure}
Iterating this process, we obtain a knot, completing the proof of of Lemma~\ref{lem:hyperbolicity}
\end{proof}

We are now ready to complete the proof of Theorem~\ref{thm:dehn}.  
Fix $X$ as in the statement of the theorem and pick a slope on each components of
 $\partial X$, say $p_{i}/q_{i}$ on the torus $T_{i} \subset \partial X$; note that we are using the
 meridian-longitudes to express the slpes as rational numbers (possibly, $1/0$).  
 Construct a 3-fold, simple cover $X^{*} \to S^{3,*}$ as in Lamme~\ref{lem:hyperbolicity}
 that corresponds to the appropriate equivalence classes
 of the slopes  (recall Notation~\ref{notation:EquivalenceClasses}).  For convenience we 
 work with $\widehat D$, the digram of $\widehat T$.

We now change the diagram $\widehat D$ by replacing the rational tangle $\widehat R_{i}$
in $B_{i}$ (that represents the equivalence class of $p_{i}/q_{i}$)  with the rational tangle $R_{i}$ 
(that realizes the slope $p_{i}/q_{i}$), $i=1,\dots,n$.  By construction the four strands of $\widehat D$
that connect to $B_{i}$ are single colored, and we color the $R_{i}$
by the same color.  Thus we obtain a diagram of a three colored link denoted $K$.

We claim that $K$ has the following properties:
	\begin{enumerate}
	\item $K$ is a knot.
	\item $K$ admits an alternating projection.
	\item This projection is non-split.
	\item This projection is strongly prime.
	\end{enumerate}
We prove each claim in order:
\begin{enumerate}
\item Since the tangles $\widehat R_{i}$ and $R_{i}$ are equivalent they connect the same points on $\partial B_{i}$
(Notation~\ref{notation:EquivalenceClasses}).  By Lemma~\ref{lem:hyperbolicity}~(7), $\widehat T$ is a knot.
Hence $K$, which is obtained from $\widehat T$ by replacing $\widehat R_{i}$ by $R_{i}$, is a knot
as well.
\item  By Lemma~\ref{lem:hyperbolicity}~(6), $\widehat D$ is alternating.  By the definition of
the equivalence classes of rational tangles, $K$ (which is obtained by replacing $\widehat R_{i}$ by $R_{i}$)
admits an alternating projection.
\item Let $\gamma \subset S^{2}$ be a simple closed curve disjoint from the diagram for $K$.  If $\gamma$ is 
diagrammatically isotopic (that is, an isotopy through curves that are transverse to the diagram at all times)
to a curve that is disjoint from
$\cup_{i} B_{i}$ then by Lemma~\ref{lem:hyperbolicity}~(2) $\gamma$
bounds a disk disjoint from $\widehat D$; this disk is also disjoint from the diagram of $K$.
If $\gamma$ is diagrammatically isotopic into $B_{i}$, then $\gamma$  
bounds a disk disjoint from the diagram for $K$ since rational
tangles are prime.  Finally, if $\gamma$ is not isotopic into or out of $\cup_{i} B_{i}$, we violate
Condition~(4b) of Lemma~\ref{lem:hyperbolicity}.  Hence the diagram for $K$
is non-split.
\item This is very similar to~(3) and is left to the reader.
\end{enumerate}

By Menasco and Thurston (see Corollary~\ref{cor:hyperbolic}), $K$ is hyperbolic. 

Next we note that the 3-coloring of $K$ defines a $3$-fold cover of $S^{3}$; by construction, the cover
of $S^{3,*}$ is $X^{*}$.  The cover of each rational tangle is disconnected and consists of a solid
torus attached to $T_{i} \subset \partial X$ with slope $p_{i}/q_{i}$, and a ball attached to a component
of $\partial X^{*} \setminus \partial X$.  Thus we obtain $X(p_{1}/q_{1},\dots,p_{n}/q_{n})$ as a simple $3$-fold 
cover of $S^{3}$ branched over $K$.

\medskip 
We now isotope each rational tangle $R_{i}$ to realize its depth, that is, realizing the twist number of each rational 
tangle (recall Subsection~\ref{subsec:twist}).
The twist number of $R_{i}$ is exactly $\depth[p_{i}/q_{i}]$.  The tangle $T$ 
(which is the projection of $K$ outside $\cup_{i} B_{i}$) has a fixed number of twist regions, say $t$.
Hence the total number of twist regions is $t + \sum_{i=1}^{n} \depth[p_{i}/q_{i}] = t + \depth[\alpha]$ 
(where $\alpha=(\alpha_{1},\dots,\alpha_{n})$ denotes the multislope on $\partial X$, as in Section~\ref{sec:intro}).  
This gives an upper bound for the twist number
for $K$: 
$$t(K)\leq t + \depth[\alpha].$$  
Lackenby~\cite{lackenby} (recall Subsection~\ref{subsec:twist}) showed that there exists a constant $c$ so that:
$$\vol[S^{3} \setminus K] \leq c t(K).$$
Hence we get:
\begin{eqnarray*}
\kvsd[3 X(\alpha_{1},\dots,\alpha_{n}]) &\leq & 3 \vol[S^{3} \setminus K] \\
	&\leq& 3c t(K) \\
	&\leq& 3c t + 3c(\depth[\alpha]). \\
\end{eqnarray*}
By setting $A = 3c$ and $B = 3ct$, we obtain constants fulfilling the requirements of Theorem~\ref{thm:dehn} 
that are valid for any multislope $\alpha' = (\alpha_{1}',\dots,\alpha_{n}')$, with $\alpha_{i}'$ in the same
equivalence class as $\alpha_{i}$.  As there are only finitely many (specifically,
$6^{n}$) equivalence classes, taking the maximal constants $A$ and $B$ for these classes 
completes the proof of the theorem.


\begin{thebibliography}{10}

\bibitem{alex}
James~W. Alexander.
\newblock Note on {R}iemann spaces.
\newblock {\em Bull. Amer. Math. Soc.}, 26(8):370--372, 1920.

\bibitem{BenedettiPetronio}
Riccardo Benedetti and Carlo Petronio.
\newblock {\em Lectures on hyperbolic geometry}.
\newblock Universitext. Springer-Verlag, Berlin, 1992.

\bibitem{blair}
Ryan~C. Blair.
\newblock Alternating augmentations of links.
\newblock {\em J. Knot Theory Ramifications}, 18(1):67--73, 2009.

\bibitem{CaoMeyehoff}
Chun Cao and G.~Robert Meyerhoff.
\newblock The orientable cusped hyperbolic $3$-manifolds of minimum volume.
\newblock {\em Invent. Math.}, 146(3):451--478, 2001.

\bibitem{CassonGordon}
A.~J. Casson and C.~McA. Gordon.
\newblock Reducing {H}eegaard splittings.
\newblock {\em Topology Appl.}, 27(3):275--283, 1987.

\bibitem{feighn}
Mark~E. Feighn.
\newblock Branched covers according to {J}. {W}. {A}lexander.
\newblock {\em Collect. Math.}, 37(1):55--60, 1986.

\bibitem{GordonLuecke}
C.~McA. Gordon and J.~Luecke.
\newblock Knots are determined by their complements.
\newblock {\em J. Amer. Math. Soc.}, 2(2):371--415, 1989.

\bibitem{hilden}
Hugh~M. Hilden.
\newblock Every closed orientable {$3$}-manifold is a {$3$}-fold branched
  covering space of {$S^{3}$}.
\newblock {\em Bull. Amer. Math. Soc.}, 80:1243--1244, 1974.

\bibitem{hlm1}
Hugh~M. Hilden, Mar{\'{\i}}a~Teresa Lozano, and Jos{\'e}~Mar{\'{\i}}a
  Montesinos.
\newblock The {W}hitehead link, the {B}orromean rings and the knot {$9_{46}$}
  are universal.
\newblock {\em Collect. Math.}, 34(1):19--28, 1983.

\bibitem{hlm2}
Hugh~M. Hilden, Mar{\'{\i}}a~Teresa Lozano, and Jos{\'e}~Mar{\'{\i}}a
  Montesinos.
\newblock On knots that are universal.
\newblock {\em Topology}, 24(4):499--504, 1985.

\bibitem{KR}
Tsuyoshi Kobayashi and Yo'av Rieck.
\newblock A linear bound on the tetrahedral number of manifolds of bounded
  volume (after {J}\o rgensen and {T}hurston).
\newblock {\em Contemporary Mathematics}, 560:27--42, 2011.

\bibitem{kojima}
Sadayoshi Kojima.
\newblock Isometry transformations of hyperbolic {$3$}-manifolds.
\newblock {\em Topology Appl.}, 29(3):297--307, 1988.

\bibitem{lackenby}
Marc Lackenby.
\newblock The volume of hyperbolic alternating link complements.
\newblock {\em Proc. London Math. Soc. (3)}, 88(1):204--224, 2004.
\newblock With an appendix by Ian Agol and Dylan Thurston.

\bibitem{lickorish}
W.~B.~Raymond Lickorish.
\newblock {\em An introduction to knot theory}, volume 175 of {\em Graduate
  Texts in Mathematics}.
\newblock Springer-Verlag, New York, 1997.

\bibitem{menasco}
W.~Menasco.
\newblock Closed incompressible surfaces in alternating knot and link
  complements.
\newblock {\em Topology}, 23(1):37--44, 1984.

\bibitem{moise}
Edwin~E. Moise.
\newblock Affine structures in {$3$}-manifolds. {V}. {T}he triangulation
  theorem and {H}auptvermutung.
\newblock {\em Ann. of Math. (2)}, 56:96--114, 1952.

\bibitem{montesinos}
Jos{\'e}~M. Montesinos.
\newblock A representation of closed orientable {$3$}-manifolds as {$3$}-fold
  branched coverings of {$S^{3}$}.
\newblock {\em Bull. Amer. Math. Soc.}, 80:845--846, 1974.

\bibitem{MotegiRieck}
Kimihiko Motegi and Yo'av Rieck.
\newblock Short geodesic after dehn filling.
\newblock In preparation.

\bibitem{myers}
Robert Myers.
\newblock Simple knots in compact, orientable {$3$}-manifolds.
\newblock {\em Trans. Amer. Math. Soc.}, 273(1):75--91, 1982.

\bibitem{NeumannZagier}
Walter~D. Neumann and Don Zagier.
\newblock Volumes of hyperbolic three-manifolds.
\newblock {\em Topology}, 24(3):307--332, 1985.

\bibitem{JairYoav}
Jair Remigio-Ju\'arez and Yo'av Rieck.
\newblock The link volumes of some prism manifolds.
\newblock Accepted to {\it Algebraic and Geometric Topology}, 2012.

\bibitem{rolfsen}
Dale Rolfsen.
\newblock {\em Knots and links}, volume~7 of {\em Mathematics Lecture Series}.
\newblock Publish or Perish Inc., Houston, TX, 1990.
\newblock Corrected reprint of the 1976 original.

\bibitem{thurston}
William~P Thurston.
\newblock The geometry and topology of 3-manifolds.
\newblock available from msri.org.

\end{thebibliography}
\end{document}